\documentclass[a4paper, reqno]{amsart}
\usepackage{amsmath, amssymb, eucal, amscd, amstext, enumerate, mathrsfs, yfonts, amsfonts,comment, color}
\usepackage[plainpages=false,colorlinks,hyperindex,pdfpagemode=None,bookmarksopen,linkcolor=red,citecolor=blue,urlcolor=blue]{hyperref}
\usepackage{cite}

\newtheorem{thm}{Theorem} 
\newtheorem{lem}[thm]{Lemma}

\newtheorem{prop}[thm]{Proposition}
\newtheorem{cor}[thm]{Corollary}

\theoremstyle{definition}
\newtheorem{defn}[thm]{Definition}

\newtheorem{problem}[thm]{Problem}

\theoremstyle{remark}

\newtheorem{rem}[thm]{Remark}

\newtheorem{eg}[thm]{Example}

\numberwithin{equation}{section}

\newcommand{\loc}{{\rm loc}}
\newcommand{\smnoind}{\noindent}

\newcommand{\CS}{\mathcal{S}}

\newcommand{\BC}{\mathbb{C}}

\newcommand{\BN}{\mathbb{N}}

\newcommand{\CC}{\mathcal{C}}
\newcommand{\CF}{\mathcal{F}}

\newcommand{\CP}{\mathcal{P}}

\newcommand{\KB}{\mathfrak{B}}
\newcommand{\KI}{\mathfrak{I}}

\newcommand{\KF}{\mathfrak{F}}

\newcommand{\KC}{\mathfrak{C}}
\newcommand{\KP}{\mathfrak{P}}
\newcommand{\KQ}{\mathfrak{Q}}
\newcommand{\KA}{\mathfrak{A}}

\newcommand{\rp}{\mathrm{p}}
\newcommand{\rz}{\mathrm{z}}
\newcommand{\dist}{\mathrm{dist}}

\newcommand{\f}{\mathbf{f}}
\newcommand{\scs}{{\sigma\mathbf{c}\sigma}}
\newcommand{\supp}{\od{\mathrm{supp}}\ \!}

\newcommand{\smf}{\mathrm{sf}}

\newcommand{\mc}{\mathrm{c}}

\newcommand{\s}{\mathrm{sp}}

\newcommand{\CPO}{C(X)_+^\mathrm{sp}}

\newcommand{\od}{\widetilde}

\begin{document}

\title[A variant of Tingley's problem]{On a variant of Tingley's problem for some function spaces}

\author[Leung, Ng and Wong]{Chi-Wai Leung \and Chi-Keung Ng \and Ngai-Ching Wong}

\address[Chi-Wai Leung]{Department of Mathematics, The Chinese University of Hong Kong, Hong Kong, China}
\email{cwleung@math.cuhk.edu.hk}

\address[Chi-Keung Ng]{Chern Institute of Mathematics and LPMC, Nankai University, Tianjin 300071, China.}
\email{ckng@nankai.edu.cn}

\address[Ngai-Ching Wong]{Department of Applied Mathematics, National Sun Yat-sen University, Kaohsiung, 80424, Taiwan, \and
School of Mathematical Sciences, Tiangong University, Tianjin 300387, China.}
\email{wong@math.nsysu.edu.tw}

\date{\today}



\begin{abstract}
Let $(\Omega, \mathfrak{A}, \mu)$ and $(\Gamma, \mathfrak{B}, \nu)$ be two arbitrary measure spaces, and $p\in [1,\infty]$.
Set
$$L^p(\mu)_+^\mathrm{sp}:= \{f\in L^p(\mu): \|f\|_p =1; f\geq 0\ \mu\text{-a.e.} \}$$
i.e., the positive part of the unit sphere of $L^p(\mu)$.
We show  that every metric preserving bijection $\Phi: L^p(\mu)_+^\mathrm{sp} \to L^p(\nu)_+^\mathrm{sp}$ can be extended (necessarily uniquely) to an isometric order isomorphism
from $L^p(\mu)$ onto $L^p(\nu)$.
A Lamperti form, i.e., a weighted composition like form, of $\Phi$ is provided, when $(\Gamma, \mathfrak{B}, \nu)$ is localizable (in particular, when it is $\sigma$-finite).

On the other hand, we show that for compact Hausdorff spaces $X$ and $Y$, if $\Phi$  is a metric preserving bijection
from the positive part of the unit sphere of $C(X)$ to that of $C(Y)$, then there is a homeomorphism $\tau:Y\to X$ satisfying $\Phi(f)(y) = f(\tau(y))$ ($f\in C(X)_+^\mathrm{sp}; y\in Y$).
\end{abstract}

\maketitle

\section{Introduction}

Tingley's problem (\cite{Tingley}) asks whether every metric preserving bijection between the unit spheres of two real Banach spaces extends to a bijective isometry between the two Banach spaces.
The problem has been studied by many people (see \cite{FP, FP17,FP18, FP18a,Ding09, Ding-C(X),Ding-L-infty,FW,Mori,MO,Tan,Tingley}).
In his survey paper \cite{Peralta18}, Peralta proposes a variant of Tingley's problem.
More precisely, for   an ordered Banach space $E$ with a generating cone, denote by
$$
E^\s_+ :=\{x\in E: x\geq 0\ \text{and}\ \|x\|=1\}
$$
the positive part of the unit sphere, or the \emph{positive unit sphere} in short, of $E$.
The problem of Peralta states as follows.

\begin{problem}\label{prob:1}
Let $\Phi: E^\s_+\to F_+^\s$ be a metric preserving bijection between the positive unit spheres of two ordered Banach spaces; namely,
$$
\|\Phi(x)-\Phi(y)\| = \|x-y\| \qquad (x,y\in E^\s_+).
$$
Is there a surjective positive linear isometry $\widetilde{\Phi}: E\to F$  extending $\Phi$?
\end{problem}

Problem \ref{prob:1} has a positive answer
when $E=F$  is the space $C_p(H)$ of the Schatten $p$-class operators
 on a complex Hilbert space $H$ for $p\in [1,\infty]$ (see \cite{MT03, MN12, Nagy13,Nagy18, Peralta19}).
This motivates Peralta  to restate in \cite[Problem 7]{Peralta18}
the following more specific  open problem of Mori (see \cite[Problem 6.3]{Mori}).

\begin{problem}\label{prob:2}
  Let $M$ and $N$ be von Neumann algebras and $p\in [1,\infty]$.  Let $\Phi: L^p(M)^\s_+ \to L^p(N)^\s_+$ be a metric preserving bijection
  between the positive unit spheres of two noncommutative $L^p$-spaces.  Is there a surjective linear isometry $\widetilde{\Phi}: L^p(M)\to
  L^p(N)$ extending $\Phi$?
\end{problem}

Let us state   what is known about Problem \ref{prob:2}:
\begin{itemize}
	\item It has a positive answer for $p=1$ (see \cite[Theorem 5.11(a)]{Mori} and \cite[Proposition 10]{LaNW-sf}).
	
	\item It has a positive answer for $p=2$ (this is precisely \cite[Proposition 3.7]{LNW-tran-prob}, because $L^2(M)_+$ is the self-dual cone in the standard form of $M$).
	
	\item It has a positive answer when $M=N=B(H)$, since then $L^p(B(H))=C_p(H)$ for all $p\in [1,\infty]$ (see \cite{MT03, MN12, Nagy13,Nagy18, Peralta19}).
\end{itemize}
Furthermore, backing up by our previous work  on non-commutative $L^p$-spaces in \cite{LeNW-Lp-shell}, it is believed that
Problem \ref{prob:2} should have a positive answer in general.
In fact, if we define
$L^p(M)_+^{[1-\epsilon,1]}: =\{ x\in L^p(M): x \geq 0\ \text{and}\ 1-\epsilon \leq \|x\|\leq 1\}$, then for a  semi-finite von Neumann $M$ and for $\epsilon\in (0,1]$, every metric preserving bijection $\Phi:L^p(M)_+^{[1-\epsilon,1]}\to L^p(N)_+^{[1-\epsilon,1]}$ extends to a surjective linear isometry between the whole noncommutative $L^p$-spaces (see \cite[Theorem 1.2]{LeNW-Lp-shell}).

In this paper, we will show that Problem \ref{prob:2} has a positive answer in the case when both $M$ and $N$ are  commutative,
namely the $L^p$-spaces for $p\in [1,\infty]$
(see Theorems \ref{thm:Lp} and \ref{thm:L-infty}), and we will also solve an analogue question for the
$C(X)$ spaces (see Theorem \ref{thm:C(X)}).
More precisely, we will present the concrete forms of
metric preserving bijections from $L^p(\mu)_+^\s$ onto $L^p(\nu)_+^\s$  (where $p\in [1, \infty]$) and the ones from
$C(X)_+^\s$ onto $C(Y)_+^\s$, respectively.
Actually, Theorems \ref{thm:Lp} and \ref{thm:L-infty} are slightly more general results than the solution for the commutative version of Problem \ref{prob:2}, because the $L^\infty(\mu)$-space  is not necessarily  a commutative von Neumann algebra unless
the underlying measure space $(\Omega,\KA,\mu)$ is localizable.

Concerning the proofs for Theorem \ref{thm:Lp}, we note that the argument in \cite{LeNW-Lp-shell} makes use of the fact that
the subset
$L^p(M)_+^{[1-\epsilon,1]}$ contains
many line segments.
However,  the positive  unit sphere $L^p(\mu)^\s_+$ contains no line segment at all if $p\in (1,\infty)$.
Hence, the argument in \cite{LeNW-Lp-shell} cannot be applied.
On the other hand, although there exist positive solutions for the original Tingley's problem for the whole
unit spheres of $L^p$-spaces and  $C(X)$ spaces
 (see \cite{Ding-C(X),Ding-L-infty,FW,Tan}), the techniques used in these solutions cannot
 be applied to the cases of the positive  unit spheres.

We will define  the notations and terminologies used in this paper, and collect some preliminary results in Section \ref{s:2}.
In particular, we will present a form of
the Radon-Nikodym theorem for localizable measure spaces, which we need in proving our results.
We will then solve, in Section \ref{s:3}, the positive type Tingley problem (i.e., Problem \ref{prob:1}) in the case when $p\in [1,\infty)$, $E=L^p(\mu)$ and $F=L^p(\nu)$.
The cases of Problem \ref{prob:1} for $E=C(X)$ and $F=C(Y)$ as well as for $E = L^\infty(\mu)$ and $F=L^\infty(\nu)$
will be solved in Section \ref{s:4}.

\section{Notations and preliminaries}\label{s:2}

\subsection{Some measure theory}
We  present some notations and preliminary results from  measure theory that are needed in the our discussion.
Let $(\Omega, \KA, \mu)$ be a measure space, where $\KA$ is a $\sigma$-algebra of subsets of the underlying set $\Omega$,
 and $\mu: \KA\to [0,\infty]$ is a measure; i.e.,
a countably additive function sending the empty set $\emptyset$ to zero.
Denote
\begin{gather}
\nonumber
\KA_\mu^0 := \{A\in \KA: \mu(A) = 0 \}, \quad  \KA_\mu^\f:=\{A\in \KA: \mu(A)<\infty \}, \\
\label{eqt:def-A-sigma}
\KA_\mu^\sigma:= \left\{{\bigcup}_{k\in \BN} A_k: A_1,A_2, ... \in \KA_\mu^\f \right\}
\quad\text{and}\quad \KA_\mu^\scs := \KA_\mu^\sigma \cup \{\Omega \setminus A: A\in \KA_\mu^\sigma \}.
\end{gather}
When the measure $\mu$ is understood, we will write $\KA^0, \KA^\f, \KA^\sigma$ and $\KA^\scs$ instead of $\KA^0_\mu, \KA^\f_\mu, \KA^\sigma_\mu$ and $\KA^\scs_\mu$ for simplicity.
Clearly, $\KA^\sigma$ and $\KA^\scs$ are, respectively, the $\sigma$-subring of $\KA$ and the $\sigma$-subalgebra of $\KA$, generated by $\KA^\f$.

If $E,F\in \KA$, we write
\begin{align}\label{eqt:def-equiv}
E\preceq F\quad&\text{when}\quad \mu(E\setminus F) = 0, \nonumber
\intertext{and}
E\equiv F\quad &\text{when} \quad E\preceq F\text{ and }F\preceq E.
\end{align}
We say that $E\in \KA$ is an \emph{essential supremum}  of a family $\{E_i \}_{i\in \KI}$ in $\KA$ if
	\begin{itemize}
	\item $E_i \preceq E$, for every $i\in \KI$;
	\item for any $F\in \KA$ satisfying $E_i \preceq F$ ($i\in \KI$), one has $E \preceq F$.
\end{itemize}
Notice that the ordinary union of a countable family in $\KA$ is an essential supremum for this family.

Let  $(\Omega, \KA, \mu)$ and $(\Gamma, \KB, \nu)$ be measure spaces.
Let $\KP$ and $\KQ$ be subrings of $\KA$ and $\KB$, respectively.
Suppose that a map $\Lambda: \KP\to \KQ$ satisfies the following conditions:
\begin{enumerate}[\ \ (R1)]
	\item $\Lambda(A)\equiv \emptyset$ if and only if $A\equiv \emptyset$ for every $A\in \KP$;
	\item $\Lambda(A_1\setminus A_2) \equiv \Lambda(A_1)\setminus \Lambda(A_2)$ for every $A_1,A_2\in \KP$;
	\item for  any $B\in \KQ$, there is $A\in \KP$ such that $\Lambda(A) \equiv B$.
\end{enumerate}
Then it follows from Conditions (R1) and (R2) that for any $A_1,A_2\in \KP$, one has
\begin{enumerate}[\ \ (R4).]
	\item $\Lambda(A_1)\preceq \Lambda(A_2)$ if and only if $A_1\preceq A_2$,
\end{enumerate}
because $A_1\preceq A_2$ if and only if $A_1\setminus A_2\equiv \emptyset$.
This, together with Condition (R3), tells us that $\Lambda$ will preserve essential suprema up to measure zero sets if they exist.
Therefore, we have:
\begin{enumerate}[\ \ (R5).]
	\item $\bigcup_{k\in \BN} \Lambda(A_k) \equiv \Lambda\big(\bigcup_{k\in \BN} A_k\big)$, for every $A_1,A_2, ... \in \KP$ with $\bigcup_{k\in \BN} A_k\in \KP$.
\end{enumerate}
Motivated by \cite{Lampeti58}, we call a map $\Lambda: \KP\to \KQ$  satisfying
(R1), (R2) and (R3) a \emph{regular set isomorphism}.

The equivalence class of an element $E\in \KP$ under the relation $\equiv$ will be denoted by $\od{E}$.
We put
\begin{equation}\label{eqt:def-widetilde-P}
\od{\KP}:= \{\od{E}:E\in \KP \}.
\end{equation}
Obviously, $\preceq$ induces an ordering $\leq$ on $\od{\KA}$.
The set complement and the intersection will induce the corresponding operations on $\od{\KA}$, denoted by $^\mathrm{c}$ and $\wedge$, respectively.
It is easy to see that $\od{\KA}$ is a Dedekind $\sigma$-complete Boolean algebra (see \cite[Theorem 321H]{Frem}).

A regular set isomorphism $\Lambda$ will induce a unique order isomorphism $\od{\Lambda}$ from $\od{\KP}$ onto $\od{\KQ}$
satisfying
\begin{equation}\label{eqt:Lambda-0}
\od{\Lambda}(\od{E}) = \od{\Lambda(E)} \qquad (E\in \KP)
\end{equation}
(see Condition (R4)).
Conversely, any order isomorphism $\od{\Lambda}$ from $\od{\KP}$ onto $\od{\KQ}$ will induce a (not necessarily unique) regular set isomorphism $\Lambda:\KP\to \KQ$ satisfying Relation \eqref{eqt:Lambda-0}.

\begin{defn}
A measure space $(\Omega, \KA, \mu)$ is said to be
\begin{enumerate}
	\item \emph{semi-finite} if for each $E\in \KA\setminus \KA^0$, there is an element $A\in \KA^\f \setminus \KA^0$ with $A\subseteq E$;
	
	\item \emph{localizable} if it is semi-finite, and essential suprema exist for all families of elements in $\KA$. 	
\end{enumerate}
\end{defn}
Obviously, $\sigma$-finite measure spaces are localizable.
If $(\Omega, \KA, \mu)$ is localizable, then $\od{\KA}$ is a Dedekind complete lattice.

\subsection{Radon-Nikodym type theorems for localizable measure spaces}
In showing our main result in Section \ref{s:3}, we need a form of the Radon-Nikodym theorem for localizable measure spaces.
However, the versions that we found from the literature  (e.g., \cite[Theorem 234O]{Frem} as well as Theorems 2.11 and 2.12 of \cite{CV})
do not seem to apply directly to our case.
Therefore, we present a few more forms of the Radon-Nikodym theorem in this section.
These forms could be useful in other situations, and they might be already known to experts.

We start with a convention.
For any function $g:\Omega \to \BC$, we set
\begin{equation}\label{eqt:def-supp}
\mathrm{supp}\ \! g:=\{x\in \Omega: g(x)\neq 0\}.
\end{equation}

\begin{lem}\label{lem:RN}
	Let $(\Omega, \KB, \nu)$ and $(\Omega, \KC,\lambda)$ be measure spaces with $(\Omega, \KB, \nu)$ being localizable.
	Suppose that $\KB^\sigma\subseteq \KC$ and $\KC^\sigma\subseteq \KB$.
	
	\noindent
	(a) Assume $\KB^0 \subseteq \KC^0$.
	There is a $\KB$-measurable function $g: \Omega\to [0,\infty)$ with
	\begin{equation}\label{eqt:RN-deriv}
	\lambda(E) = \int_E g\ \! d\nu \qquad (E\in \KC^\sigma)
	\end{equation}
if and only if
 there exists an element $A_0\in \KB$ such that for every $E\in \KC^\sigma$, one has
	$E\cap A_0\in \KB^\sigma$ and $\lambda(E) = \lambda(E\cap A_0)$  (note that this condition holds when $\KC^\sigma \subseteq \KB^\sigma$).
	
	\noindent
	(b) Assume that there is a non-negative $\KB$-measurable function $g$ satisfying Relation \eqref{eqt:RN-deriv}.
	Then such a function $g$ is unique up to   $\nu$-measure zero sets if and only if
	for every $A\in \KB^\sigma\setminus \KB^0$, one can find $E\in \KC^\sigma\setminus \KB^0$ with $E\subseteq A$ (this condition holds when $\KB^\sigma\subseteq \KC^\sigma$).
	In this case, we  write the unique function $g$ as ${d\lambda}/{d\nu}$ and call it the \emph{Radon-Nikodym derivative} of $\mu$ respect to $\nu$.
\end{lem}
\begin{proof}
(a) The ``Only if'' part:
Take $A_0:= \mathrm{supp}\ \! g$.
Consider $E\in \KC^\f$.
Since $\int_E g\ \! d\nu = \lambda(E) < \infty$, we know that $E\cap A_0\in \KB^\sigma$.
Moreover, as $\KB^\sigma\subseteq \KC$, one has $E\cap A_0\in \KC^\f$, and hence $\lambda(E) = \int_E g\ \! d\nu = \int_{E\cap A_0} g\ \! d\nu = \lambda(E\cap A_0)$.
The general case of $E\in \KC^\sigma$ follows from the $\sigma$-additivity.
	
\noindent
The ``if'' part:
Consider $E\in \KC^\sigma$ and put
$$\KB^E:= \{A\in \KB: A\subseteq  E\cap A_0 \}.$$
As $E\cap A_0\in \KB^\sigma$, we know that $\KB^E\subseteq \KB^\sigma\subseteq \KC$ and that $\nu|_{\KB^E}$ is a $\sigma$-finite measure on $E\cap A_0$.
The hypothesis  $\KB^0\subseteq \KC^0$ means that $\lambda|_{\KB^E}$ is absolutely continuous with respects to $\nu|_{\KB^E}$.
The usual Radon-Nikodym theorem (for $\sigma$-finite measures) will then give a  $\KB^E$-measurable function $g^E: E\cap A_0\to [0,\infty)$ such that
$\lambda(A) = \int_A g^E\ \! d\nu$ ($A\in \KB^E$).
Let us extend $g^E$ to a function $\bar g^E:E\to [0,\infty)$ by setting $\bar g^E(x) := 0$ ($x\in E\setminus A_0$).
Then $\bar g^E$ is $\KB$-measurable (notice that $E\in \KC^\sigma\subseteq \KB$) and we have
$$\lambda(E) = \lambda(E\cap A_0) = \int_{E\cap A_0} g^E\ \! d\nu = \int_E \bar g^E\ \! d\nu.$$
It is not hard to check that for any $E,F\in \KC^\sigma$, one has $g^F|_{E\cap F\cap A_0} = g^E|_{E\cap F\cap A_0}$ $\nu$-a.e..
From this, we have $\bar g^F|_{E\cap F} = \bar g^E|_{E\cap F}$ $\nu$-a.e.
By \cite[Theorem 213N]{Frem}, there is a
$\KB$-measurable function $g:\Omega \to [0,\infty)$ such that for every $E\in \KC^\sigma$, we have $g|_E = \bar g^E$ $\nu$-a.e.
Hence,
$$\lambda(E) = \int_E \bar g^E\ \! d\nu  = \int_E g\ \! d\nu \qquad (E\in \KC^\sigma).$$
	
	\noindent
	(b) The ``Only if'' part:
	Assume on the contrary that there exists $A_1\in \KB^\sigma \setminus \KB^0$ such that for every $E\in \KC^\sigma$ with $E\subseteq A_1$, one has $\nu(E) = 0$ (recall that $\KC^\sigma\subseteq \KB$).
	Define $g':\Omega \to \BC$ such that $g'|_{\Omega\setminus A_1} = g|_{\Omega\setminus A_1}$ and $g'(x):=g(x) +1$ for every $x\in A_1$.
	Consider $F\in \KC^\sigma$ ($\subseteq \KB$).
	As $F\cap A_1\in \KC^\sigma$ (note that $\KB^\sigma \subseteq \KC$) and $F\cap A_1\subseteq A_1$, one knows from the above assumption that $\nu(F\cap A_1) = 0$.
	Thus,
	$$\lambda(F) = \int_{F}g \ \!d\nu  = \int_{F\setminus A_1}g \ \!d\nu = \int_{F\setminus A_1}g' \ \!d\nu = \int_{F}g' \ \!d\nu.$$
	However, $\nu\big\{ x\in \Omega: g(x)\neq g'(x) \big\} > 0$.
	
	\noindent
	The ``if'' part:
	Suppose that $g'$ is another non-negative $\KB$-measurable function satisfying Relation \eqref{eqt:RN-deriv}.
	Assume on the contrary that
	$$F:=\{x\in \Omega: g(x) < g'(x) \}\notin \KB^0.$$
	Since $(\KB,\nu)$ is semi-finite, we can find $A\in \KB^\f\setminus \KB^0$ such that $A\subseteq F$.
	The hypothesis gives $E\in \KC^\sigma\setminus \KB^0$ with $E\subseteq A$.
	From this, we know that $\int_E g \ \!d\nu = \lambda(E) = \int_E g' \ \!d\nu$.
	However, as $\nu(E)>0$ (recall that $\KC^\sigma \subseteq \KB$) and $E\subseteq F$, we have the contradiction that $\int_E g' - g \ d\nu > 0$.
	A similar contradiction arises if $\{x\in \Omega: g(x) > g'(x) \}\notin \KB^0$.
\end{proof}

\begin{rem}
(a) Observe that in Lemma \ref{lem:RN}
for the integration $\int_E g\ \! d\nu$ to have meaning for all $E\in \KC^\sigma$, we need $\KC^\sigma\subseteq \KB$.
On the other hand,  for the condition $\lambda(E\cap A_0)$ in part (a) to make sense, we need the requirement of $\KB^\sigma \subseteq \KC$.

\smnoind
(b) Using the argument of Lemma \ref{lem:RN} (replacing $\KB^E$ by $\{A\in \KB:A\subseteq E \}$ in the proof of part (a)), one can show that if $\KB^\sigma\subseteq \KC$ and $\KB^0\subseteq \KC^0$, then there exists a unique (up to $\nu$-measure zero sets) $\KB$-measurable function $g:\Omega\to [0,\infty)$ such that $\lambda(E) = \int_E g\ \! d\nu$ for every $E\in \KB^\sigma$, instead of $E\in \KC^\sigma$.
Some authors prefer to have this form of the Radon-Nikodym theorem (see, e.g., \cite[Theorem 19.27]{HS}).
However, in this case, the resulting function $g$ may not be good enough to recover the essential part of $\lambda$, because it could happen that $\KC^\sigma\setminus \KB^\sigma$ is a very large collection.
For example, let $\nu$ be the counting measure on the collection $\KB$ of all subsets of $[0,1]$ and $\lambda$ be the Lebesgue measure on the collection $\KC$ of Borel subsets on $[0,1]$.
Then $\KB^\sigma$ is the collection of all countable subsets of $[0,1]$, but countable subsets are of zero Lebesgue measure.
In this case, if $0$ is the constant zero function, then $\lambda(E) = \int_E 0\ \! d\nu$ for each $E\in \KB^\sigma$, but this function gives no information about $\lambda$.

\smnoind
(c) Equality \eqref{eqt:RN-deriv} is equivalent to the condition that
$\int f\ \! d\lambda = \int fg \ \!d\nu$ for every $\lambda$-integrable function $f$.
\end{rem}

The following is a well-known fact.

\begin{lem}\label{lem:RN>sf}
Let $(\Omega, \KB, \nu)$ and $(\Omega, \KC,\lambda)$ be measure spaces such that $(\KB, \nu)$ is semi-finite and $\KC\subseteq \KB$.
Suppose that there is a  $\KB$-measurable function $g: \Omega\to [0,\infty)$ satisfying
\begin{equation}\label{eqt:str-RN}
\lambda(E) = \int_E g\ \! d\nu \qquad (E\in \KC).
\end{equation}
Then $(\KC, \lambda)$ is semi-finite.
If, in addition, $\KB^0\subseteq \KC$, then we have $\KB^0\subseteq \KC^0$.
\end{lem}

\smallskip

\begin{prop}\label{prop:R-N}
Let $(\Omega, \KB, \nu)$ and $(\Omega, \KC,\lambda)$ be measure spaces such that $(\Omega, \KB, \nu)$ is localizable.

\noindent
(a) Suppose that $\KB^\sigma\subseteq \KC\subseteq \KB$ (in particular, when $\KC = \KB$).
Then there is a  $\KB$-measurable non-negative function $g$ satisfying
Relation \eqref{eqt:str-RN}
if and only if the following conditions hold
\begin{enumerate}
	\item $(\KC,\lambda)$ is semi-finite;
	\item $\lambda$ is absolutely continuous with respect to $\nu$, i.e., $\KB^0 \subseteq \KC^0$;
	\item there exists $A_0\in \KB$ such that for each $E\in \KC^\sigma$, one has $E\cap A_0\in \KB^\sigma$ and $\lambda(E) = \lambda(E\cap A_0)$.
	\end{enumerate}
In this case, the function $g$ is unique up to $\nu$-measure zero sets.

\noindent
(b) Suppose that $\KC^\sigma\subseteq \KB^\sigma\subseteq \KC\subseteq \KB$.
There is a  $\KB$-measurable non-negative function $g$ satisfying Relation \eqref{eqt:str-RN}
if and only if $\KB^0\subseteq \KC^0$ and $(\KC, \lambda)$ is semi-finite.
In this case, $g$ is unique up to $\nu$-measure zero sets.

\noindent
(c) If $\KC^\sigma\subseteq \KB^\sigma\subseteq \KC$ and  $\KB^0\subseteq \KC^0$, then one can find a  $\KB$-measurable non-negative function $g$ satisfying Relation \eqref{eqt:RN-deriv}.

\noindent
(d) Suppose that $\KB^\sigma =  \KC^\sigma$.
	Then $\KB^0 \subseteq \KC^0$ if and only if there is a
	$\KB$-measurable non-negative function $g$ satisfying Relation \eqref{eqt:RN-deriv}.
In this case, $g$ is unique up to $\nu$-measure zero sets.
\end{prop}
\begin{proof}
(a) Let us first verify the uniqueness of $g$ via Lemma \ref{lem:RN}(b) (note that as Relation \eqref{eqt:str-RN} is stronger than Relation \eqref{eqt:RN-deriv}, the conclusion of Lemma \ref{lem:RN}(b) implies what we needed).
In fact, let $A\in \KB^\sigma\setminus \KB^0$.
If $\lambda(A) = 0$, then $A\in \KC^0 \setminus \KB^0$.
Otherwise, we have $A\in \KB^\sigma \setminus \KC^0 \subseteq \KC\setminus \KC^0$.
Now, the semi-finiteness of $(\KC,\lambda)$ (see Lemma \ref{lem:RN>sf}) produces $E\in \KC^\f\setminus \KC^0$ with $E\subseteq A$, and the second assertion of Lemma \ref{lem:RN>sf} implies $E\not\in \KB^0$.
In other words, the hypothesis of Lemma \ref{lem:RN}(b) is satisfied.

\noindent
The ``only if'' part:
Conditions (1) and (2) follows from Lemma \ref{lem:RN>sf}, while Condition (3) follows from Lemma \ref{lem:RN}(a).

\noindent
The ``if'' part:
It follows from Lemma \ref{lem:RN}(a) as well as Conditions (2) and (3) that there exists a  $\KB$-measurable function satisfying Relation \eqref{eqt:RN-deriv}.
Now, it follows from the semi-finiteness of $(\KC, \lambda)$  and \cite[Lemma 213A]{Frem} that Relation \eqref{eqt:str-RN} holds.

\noindent
(b)	As the inclusion $\KC^\sigma\subseteq \KB^\sigma$ implies Condition (3) of part (a), this part follows directly from part (a).

\noindent
(c) Since the inclusion $\KC^\sigma\subseteq \KB^\sigma$ implies Condition (3) of part (a), it follows from $\KB^0\subseteq \KC^0$ and Lemma \ref{lem:RN}(a) that there exists a $\KB$-measurable function $g$ satisfying Relation \eqref{eqt:RN-deriv}.

\noindent
(d)	Note that the uniqueness follows from Lemma \ref{lem:RN}(b), because $A\in\KC^\sigma$ for every $A\in \KB^\sigma\setminus \KB^0$.

\noindent
The ``only if'' part:
As  the inclusion $\KC^\sigma\subseteq \KB^\sigma$ implies Condition (3) of part (a), Lemma \ref{lem:RN}(a) applies and we obtain the required function $g$.

\noindent
The ``if'' part:
For any $E\in \KB^0\subseteq \KB^\sigma = \KC^\sigma$, one can obtain from Relation \eqref{eqt:RN-deriv} that $\lambda(E) = 0$.
\end{proof}

\smallskip

\begin{rem}
In the situation when $\KB \subseteq \KC$, we may apply part (a) above to the measure space $(\Omega, \KB, \lambda|_\KB)$.
In this case, Proposition \ref{prop:R-N}(a) is similar to \cite[Theorem 234O]{Frem}, but it seems that our presentation is more accessible and
easier to use.
\end{rem}

\section{The case of $L^p(\mu)$ where $p\in [1,\infty)$}\label{s:3}

Let $(\Omega, \KA, \mu)$ be a measure space and $p\in [1, \infty)$.
Let $L^p(\mu)$ be the Banach space of (the equivalence classes of)
absolutely $p$-integrable functions on $(\Omega, \KA, \mu)$.
As a convention, by saying $f\in L^p(\mu)$, we mean that $f$ is an absolutely $p$-integrable function representing the equivalence class of
measurable functions $g$ satisfying $\mu\{x\in\Omega : f(x)\neq g(x)\}=0$.

Let us recall that if $(\Omega, \KA_\mathrm{C}, \mu_\mathrm{C})$ is the Caratheodory extension of $(\Omega, \KA, \mu)$ (i.e., $\KA_\mathrm{C}$ is the largest $\sigma$-algebra on which the outer measure induced by $\mu$ is a measure), then $L^p(\mu)\cong L^p(\mu_\mathrm{C})$ canonically.
Therefore, one can only expect to recover $(\Omega, \KA_\mathrm{C}, \mu_\mathrm{C})$ completely from $L^p(\mu)$ at best.
However, notice that $\od{\KA^\sigma} \cong \od{\KA_\mathrm{C}^\sigma}$ (see \eqref{eqt:def-widetilde-P} for the notation).

For the clarity with notations, we set
\begin{equation*}
\supp f := \od{\mathrm{supp}\ \! f}  \qquad (f\in L^p(\mu));
\end{equation*}
that is, $\supp f$ is the equivalence class of the measurable sets represented by
$\mathrm{supp}\ \! f$ (see \eqref{eqt:def-supp}) under the relation $\equiv$ defined in \eqref{eqt:def-equiv}.
Note that if $g$ is another function in the equivalence class represented by $f$, then $\od{\mathrm{supp}\ \! g} = \od{\mathrm{supp}\ \! f}$,
and thus $\supp f$ is well-defined.
Moreover, one has
\begin{equation}\label{eqt:KA-sigma}
\od{\KA^\sigma} = \{\supp f: f\in L^p(\mu)\}, \quad \text{which implies} \quad L^p(\mu) \cong L^p(\mu|_{\KA^\scs})\ \text{canonically}.
\end{equation}
These two statements are not true when $p=\infty$.

As in the introduction, we set
$L^p(\mu)_+^\s$ to be the positive unit sphere of $L^p(\mu)$.
Moreover, we put
\begin{align*}
L^p_F(\mu)_{+}^\s := \{f\in L^p(\mu)_+^\s: \supp f \leq \od{F} \}\qquad (F\in \KA),
\end{align*}
(where $\leq$ is the ordering of the lattice $\od{\KA}$),
and define
$$\dist\big(f, L^p_F(\mu)_+^\s\big):= \inf\big\{\|f-g\|_p: g\in  L^p_F(\mu)_+^\s\big\}.$$

\begin{lem}\label{lem:dist}
If $f\in L^p(\mu)_+^\s$ and $F\in \KA$, then
$$
\dist\big(f, L^p_F(\mu)_+^\s\big) = 1 - \|f|_F\|_p^p + (1 - \|f|_F\|_p)^p.
$$
\end{lem}
\begin{proof}
Notice that if $\|f|_F\|_p^p = 0$, then $f\in L^p_{\Omega\setminus F}(\mu)_+^\s$, and we have $\|f - h\|_p^p =  \|f\|_p^p + \|h\|_p^p = 2$ for every $h\in L^p_F(\mu)_+^\s$, which gives the asserted equality.
In the following, we assume that $\|f|_F\|_p^p > 0$.
If we set $f_{F,1}:= \frac{f|_F}{\| f|_F\|_p}$, then $f_{F,1}\in L^p_F(\mu)_+^\s$ and
$$\big\|f - f_{F,1}\big\|_p^p = 1 - \big\| f|_{F}\big\|_p^p + \big\| f|_F - f_{F,1}\big\|_p^p = 1 - \big\| f|_{F}\big\|_p^p + \big(1 - \big\|f|_F\big\|_p\big)^p.$$
On the other hand, if $g\in L^p_F(\mu)_+^\s$, then it follows from $\| f|_{\Omega\setminus F}\|_p^p + \|f|_F\|_p^p = 1$ that
$$
\|f - g\|_p^p = \big\| f|_{\Omega\setminus F}\big\|_p^p + \big\|f|_F - g\big\|_p^p \geq 1 - \big\| f|_{F}\big\|_p^p + \big(1 - \big\|f|_F\big\|_p\big)^p.
$$
This completes the proof.
\end{proof}

Let $(\Omega, \KA, \mu)$ and $(\Gamma, \KB, \nu)$ be two measure spaces.
Suppose that $\Lambda: \KA^\sigma\to \KB^\sigma$ is a regular set isomorphism.
We define an extended real-valued  function $\mu^\Lambda:\KB^\sigma\to [0,  \infty]$ as follows:
\begin{equation}\label{eqt:def-mu-Lambda}
\mu^\Lambda(B) := \mu(A) \quad \text{when }A\in \KA^\sigma\text{ satisfying } \Lambda(A)\equiv B.
\end{equation}
Note that such an element $A$ always exists, because of Condition (R3).
Moreover, $\mu^\Lambda(B)$ is independent of the choice of $A$ (by Conditions (R1) and (R2)).
These show that $\mu^\Lambda$ is well-defined.
Furthermore, $\mu^\Lambda$ is $\sigma$-additive, due to Condition (R5).
We extends $\mu^\Lambda$ to a function on the $\sigma$-subalgebra $\KB^\scs\subseteq \KB$ generated by $\KB^\sigma$ (see \eqref{eqt:def-A-sigma}) through
\begin{equation}\label{eqt:def-mu-Lambda2}
\mu^\Lambda(F) = \infty \qquad (F\in \KB^\scs\setminus \KB^\sigma).
\end{equation}
It is easy to check that $\mu^\Lambda$ is a measure.
On the other hand, if we put
$$\Lambda(\Omega \setminus E):= \Gamma\setminus \Lambda(E) \qquad (E\in \KA^\sigma),$$
then it is not hard to check that $\Lambda :\KA^\scs \to \KB^\scs$  is a regular set isomorphism (note that one has $\Lambda(A_1\cap A_2) \equiv \Lambda(A_1)\cap \Lambda(A_2)$, because of Conditions (R2), (R4) and (R5)).
In view of \eqref{eqt:def-mu-Lambda}, we obtain an isometric order isomorphism
$\widehat{\Lambda}:L^p(\mu) \cong L^p(\mu|_{\KA^\scs}) \to L^p(\mu^\Lambda)$ for $p\in [1,\infty)$ such that
$$\widehat{\Lambda} (\mathbf{1}_{A}) = \mathbf{1}_{\Lambda(A)} \qquad (A\in \KA^\f);$$
here, $\mathbf{1}_{A}$ is the indicator function for the subset $A$ of $\Omega$.

\begin{thm}\label{thm:Lp}
Let $p\in [1,\infty)$, and let $(\Omega, \KA, \mu)$ and $(\Gamma, \KB, \nu)$ be  arbitrary measure spaces.
Suppose that $\Phi: L^p(\mu)_+^\s \to L^p(\nu)_+^\s$ is a metric preserving bijection.

\noindent
(a) Then $\Phi$ can be extended (uniquely) to an isometric order isomorphism from $L^p(\mu)$ onto $L^p(\nu)$.

\noindent
(b) There is a regular set isomorphism $\Lambda:\KA^\sigma \to \KB^\sigma$ such that for each $E\in \KA^\sigma$, one can find a $\KB^\scs$-measurable function $h^E:\Lambda(E)\to [0,\infty)$,  satisfying
\begin{equation*}\label{eqt:h-E-compat}
h^F|_{\Lambda(E)} = h^E \quad \nu\text{-a.e.}, \quad \text{for } E,F\in \KA^\sigma \text{ with }E\preceq F,
\end{equation*}
such that
$$
\Phi(g) = \widehat{\Lambda}(g) h^{E}\quad \nu\text{-a.e.}, \quad \text{for }  g\in L^p(\mu)_+^\s \text{ with } \supp g\leq \od{E}.
$$

\noindent
(c) If $(\Gamma, \KB, \nu)$ is localizable, then the function ${d\mu^\Lambda}/{d\nu}$ as in Lemma \ref{lem:RN} exists and we have
\begin{equation}\label{eqt:Lamp-form}
\Phi(g) = \widehat{\Lambda}(g) \Big(\frac{d\mu^\Lambda}{d\nu}\Big)^{1/p}\ \ \nu\text{-a.e.}\qquad (g\in L^p(\mu)_+^\s).
\end{equation}
If, in addition, $(\Omega, \KA, \mu)$ is also localizable, then $\Lambda$ extends to a regular set isomorphism from $\KA$ to $\KB$.
\end{thm}
\begin{proof}
As the argument for part (a) requires part (b), and the argument for part (b) requires part (c), we will first establish part (c).
Then we will prove part (b) before verifying part (a).

\noindent
(c) Let us first recall from Section \ref{s:2} that the set $\od{\KA}$ of equivalence classes as in \eqref{eqt:def-widetilde-P} is a Boolean algebra with its complement being denoted by $^\mathrm{c}$.
Moreover, as in Relation \eqref{eqt:KA-sigma}, one has
$$\od{\KA^\sigma}\setminus \{\od{\emptyset} \} = \{\supp g: g\in L^p(\mu)_+^\s\}.$$
For any $A\in \KA^\sigma\setminus \KA^0$, choose $g_0\in L^p(\mu)_+^\s$ with $\supp g_0 = \od{A}$.
If $f\in L^p(\mu)_+^\s$ satisfying $\supp f\leq \od{A}^\mathrm{c}$, we know that $\|f - g_0\|_p^p = \|f\|_p^p + \|g_0\|_p^p = 2$, and hence, $\|\Phi(f) - \Phi(g_0)\|_p^p = 2$, which implies $\supp \Phi(f) \leq \supp \Phi(g_0)^\mathrm{c}$.
Since the converse also holds, we have
$$
\supp f \leq (\supp g_0)^\mc\quad\text{if and only if}\quad \supp \Phi(f) \leq \supp \Phi(g_0)^\mc.
$$
Suppose that $\od{A} = \supp g_1$ for another function $g_1\in L^p(\mu)_+^\s$.
If $\supp \Phi(g_0)\wedge \supp\Phi(g_1)^\mc\neq \od{\emptyset}$, then as $\supp \Phi(g_0)\wedge \supp\Phi(g_1)^\mc\in \od{\KB^\sigma}$, by the surjectivity of $\Phi$, we can find $f \in L^p(\mu)_+^\s$ with
$$\supp \Phi(f) = \supp \Phi(g_0)\wedge \supp\Phi(g_1)^\mc.$$
Therefore,	
$\supp f\leq (\supp g_1)^\mc = (\supp g_0)^\mc$.
However, this gives a contradiction that  $\supp \Phi(f) \leq \supp \Phi(g_0)^\mc$.
Thus, $\supp \Phi(g_0)\leq \supp\Phi(g_1)$.
Similarly, $\supp \Phi(g_1)\leq \supp\Phi(g_0)$, and we know that $\supp \Phi(g_0)$ depends only on $\od{A}$.
From this, as well as the corresponding fact for $\od{\KB^\sigma}\setminus \{\od{\emptyset} \}$, we obtain a bijection $\od{\Lambda}$ from $\od{\KA^\sigma}$ onto $\od{\KB^\sigma}$ such that
\begin{equation}\label{eqt:disj-pres}
\od{\Lambda}(\alpha_1)\wedge \od{\Lambda}(\alpha_2) = \od{\emptyset} \quad \text{if and only if}\quad  \alpha_1\wedge \alpha_2 = \od{\emptyset} \qquad (\alpha_1, \alpha_2\in \od{\KA^\sigma}),
\end{equation}
and that
$\od{\Lambda}\big(\supp g\big) = \supp \Phi(g)$ ($g\in L^p(\mu)_+^\s$).

Let $\alpha_1,\alpha_2\in \od{\KA^\sigma}$.
Then $\alpha_1 \leq \alpha_2$ if and only if for every $\alpha_3\in \od{\KA^\sigma}$ with $\alpha_2\wedge \alpha_3 = \od{\emptyset}$, one has $\alpha_1 \wedge \alpha_3 = \od{\emptyset}$ (observe that $\alpha_1 \wedge \alpha_2^\mc\leq \alpha_1\in \od{\KA^\sigma}$).
From this, we know that $\od{\Lambda}$ is order preserving.
By considering also $\Phi^{-1}$, one concludes that $\od{\Lambda}$ is an order isomorphism.
Hence, $\od{\Lambda}$ produces a regular set isomorphism $\Lambda: \KA^\sigma \to \KB^\sigma$ satisfying Relation \eqref{eqt:Lambda-0}, and we have
\begin{equation}\label{eqt:Lambda-Phi}
\od{\Lambda(E)} = \supp \Phi(f) \quad \text{ if $E\in \KA^\sigma$ and $f\in L^p(\mu)_+^\s$ with }\od{E} = \supp f.
\end{equation}
Note that the above holds for any measure spaces $(\Omega, \KA, \mu)$ and $(\Gamma, \KB, \nu)$ with no extra assumption.

Now, we consider the case when $(\Gamma,\KB,\nu)$ is localizable.
If we set $(\KC, \lambda) := (\KB, \mu^\Lambda)$, then $\KC^\sigma = \KB_{\mu^\Lambda}^{\sigma} = \KB_{\nu}^{\sigma}$, since $\Lambda$ is a set regular isomorphism from $\KA^\sigma$ to $\KB^\sigma$ (recall that $\KB^\sigma$ is the shorthanded form for $\KB_{\nu}^{\sigma}$).
Observe that $\mu^\Lambda(B) = 0$ if and only if $\nu(B) = 0$ (because of Condition (R1)).  By Proposition \ref{prop:R-N}(d), the Radon-Nikodym derivative ${d\mu^\Lambda}/{d\nu}: \Gamma \to [0,\infty)$ exists.
In particular,
$$\mu^\Lambda(\Lambda(A)) = \int_{\Lambda(A)} \frac{d\mu^\Lambda}{d\nu}\ \! d\nu \qquad (A\in \KA^\sigma).$$

Suppose that $\beta:= \supp d\mu^\Lambda/d\nu\neq \od{\Gamma}$.
As $(\KB, \nu)$ is semi-finite, there is $B\in \KB^\f\setminus \KB^0$ with $\od{B}\leq \beta^\mc$, and by Condition (R3), we may assume that $B = \Lambda(A)$ for some $A\in \KA^\sigma$.
Hence, $\mu^\Lambda(B) = \int_B \frac{d\mu^\Lambda}{d\nu}\ \!d\nu = 0$.
By the definition of $\mu^\Lambda$, we have $\mu(A) = 0$, and this  gives the contradiction that  $\nu(B) = 0$ (because of Condition (R1)).
Consequently,
\begin{equation}\label{eqt:non-zero-a.e.}
\left(\supp \frac{d\mu^\Lambda}{d\nu} \right)^\mc = \od{\emptyset}.
\end{equation}

On the other hand, it is clear that
$$g\mapsto g \left(\frac{d\mu^\Lambda}{d\nu}\right)^{1/p}$$
is an isometry from $L^p(\mu^\Lambda)$ to $L^p(\nu)$.
Define $\Psi: L^p(\mu)_+^\s \to L^p(\nu)_+^\s$ by $\Psi(f) = \widehat{\Lambda}(f) (d\mu^\Lambda/d\nu)^{1/p}$.
Then $\Psi$ is the restriction of a linear isometry from $L^p(\mu)$ to $L^p(\nu)$.
It remains to show that $\Phi^{-1}\circ \Psi$ is the identity map on $L^p(\mu)_+^\s$ (since $\Phi$ will then extend to a linear isometry and this isometry is automatically an order isomorphism).

Indeed, one knows from Relations \eqref{eqt:Lambda-Phi} and \eqref{eqt:non-zero-a.e.} that
$$\supp \Psi(f)^\mc  = \od{\Lambda}(\supp f)^\mc \vee (\supp d\mu^\Lambda/d\nu)^\mc = \supp \Phi(f)^\mc \qquad (f\in L^p(\mu)_+^\s).$$
This implies that
\begin{equation}\label{eqt:supp-Psi-1-Phi}
\supp \Phi^{-1}(\Psi(g)) = \supp g \qquad (g\in L^p(\mu)_+^\s).
\end{equation}
Fix $f\in L^p(\mu)_+^\s$ and put $\check f:= \Phi^{-1}\circ\Psi(f)$.
Consider $F\in \KA$.
It follows from Relation \eqref{eqt:supp-Psi-1-Phi} that
$$\dist\big(f, L^p_F(\mu)_+^\s \big)   = \dist\big(\check f, \Phi^{-1}\circ\Psi\big(L^p_F(\mu)_+^\s\big) \big) \geq \dist\big(\check f, L^p_F(\mu)_+^\s \big),$$
and we know from Lemma \ref{lem:dist} that
$$1 - \big\|f|_F\big\|_p^p + \big(1 - \big\|f|_F\big\|_p\big)^p \geq 1 - \big\|\check f|_F\big\|_p^p + \big(1 - \big\|\check f|_F\big\|_p\big)^p.$$
From this, we have $\int_F f^p d\mu = \big\|f|_F\big\|_p^p \geq \big\|\check f|_F\big\|_p^p = \int_F \check f^p d\mu$, for any $F\in \KA$.
This gives $f \geq \check f$ $\mu$-a.e.
However, since $\int_\Omega f^p d\mu = 1 = \int_\Omega \check f^p d\mu$, we know that $f = \check f$ $\mu$-a.e. as required.

Now, assume that both $(\KA,\mu)$ and $(\KB,\nu)$  are localizable.
Consider $\alpha\in \od{\KA}$, and set
$$\CC_\alpha:= \{\beta\in \od{\KA^\sigma}: \beta\leq \alpha \}.$$
The semi-finiteness of $(\KA,\mu)$ implies that $\alpha= \bigvee \CC_\alpha$.
We extend $\od{\Lambda}$ to a map from $\od{\KA}$ to $\od{\KB}$ by setting $\od{\Lambda}(\alpha) := \bigvee \od{\Lambda}(\CC_\alpha)$.
Since $(\KA,\mu)$ is also localizable, we can construct from $\Phi^{-1}$ a map from $\od{\KB}$ to $\od{\KA}$ and conclude that $\od{\Lambda}$ is a bijection from $\od{\KA}$ onto $\od{\KB}$.
Moreover, one obtains from Relation \eqref{eqt:disj-pres} that
$\od{\Lambda}(\alpha_1)\wedge \od{\Lambda}(\alpha_2) = \od{\emptyset}$ if and only if $\alpha_1\wedge \alpha_2 = \od{\emptyset}$ $(\alpha_1,\alpha_2\in \od{\KA}).$
Now, the argument as in the above shows that $\od{\Lambda}$ is an order isomorphism, which produces a regular set isomorphism from $\KA$ to $\KB$.

\noindent
(b) As noted above, the argument in the first part of the proof of
part (c) works for general measure spaces, and we obtain a regular set isomorphism
$\Lambda: \KA^\sigma\to \KB^\sigma$ satisfying Relation \eqref{eqt:Lambda-Phi}.
Consider $E\in \KA^\sigma$.
We set  $\KA^E:= \{F\in \KA: F\subseteq E \}$ and $\mu|_E:= \mu|_{\KA^E}$.
By Relation \eqref{eqt:Lambda-Phi},  $\Phi$ induces a metric preserving bijection
$$\Phi^E:L^p\big(\mu|_E\big)_+^\s \to L^p\big(\nu|_{\Lambda(E)}\big)_+^\s$$
such that $\Phi^E(f) = \Phi(f)|_{\Lambda(E)}$ for any $f\in L^p(\mu|_E)_+^\s$ (which is identified with $L^p_E(\mu)_+^\s$).
Since both $(E, \KA^E, \mu|_E)$ and $\big(\Lambda(E), \KA^{\Lambda(E)}, \nu|_{\Lambda(E)}\big)$
are $\sigma$-finite and thus localizable, part (c) produces a regular set isomorphism $\Lambda^E: \KA^E\to \KB^{\Lambda(E)}$
satisfying the corresponding form of Relation \eqref{eqt:Lamp-form}.
Since
$$\od{\Lambda^E}(\supp f) = \supp \Phi^E(f) \quad \text{and} \quad  \od{\Lambda}(\supp f) = \supp \Phi(f) \qquad (f\in L^p(\mu^E)_+^\s)$$
(see Relation \eqref{eqt:Lambda-Phi}), we know that $\od{\Lambda}|_{\od{\KA^E}} = \od{\Lambda^E}$, and we may replace $\Lambda^E$ by $\Lambda|_{\KA^E}$.
In particular, when $g\in L^p(\mu)_+^\s$ with $\supp g\leq \od{E}$, we have
$\widehat{\Lambda^E}(g) = \widehat{\Lambda}(g)$.
Moreover, if we set
\begin{equation}\label{eqt:def-h-E}
h^E := \big(d\mu|_E^{\Lambda|_{\KA^E}}/d\nu|_{\Lambda(E)}\big)^{1/p},
\end{equation}
then $\Phi(g) = \widehat{\Lambda}(g) h^E$ $\nu$-a.e.
On the other hand, by the uniqueness of the Radon-Nikodym derivative,  $h^F|_{\Lambda(E)} = h^E$ $\nu$-a.e., when $E,F\in \KA^\sigma$ satisfying $E\preceq F$.

\noindent
(a) Let us define $\bar \Phi:L^p(\mu)\to L^p(\nu)$ by
$$\bar \Phi(f) = \widehat{\Lambda}(f) h^{E_0}\quad \text{ when $f\in L^p(\mu)$ and $E_0\in \KA^\sigma$ with }\supp f = \od{E_0}$$
where $\Lambda:\KA^\sigma \to \KB^\sigma$ is the map defined as in part (b) and $h^{E_0}$ is as in Relation \eqref{eqt:def-h-E}.
Then $\|\bar \Phi(f)\| = \|f\|$ for any $f\in L^p(\mu)$.
Furthermore, by part (b), $\bar \Phi$ is a bijective linear  extension of $\Phi$ sending $L^p(\mu)_+$ onto $L^p(\nu)_+$.
\end{proof}

Notice that in the case of $\sigma$-finite measure spaces, the argument of part (c) above can be simplified quite a bit.
Moreover, the argument for part (b) only requires the $\sigma$-finite case of part (c).
Nevertheless, we choose to consider the more complicated case of localizable measure spaces.
It is due to the fact that $(\Gamma, \KB,\nu)$ is localizable exactly when $L^\infty(\nu)$ is a von Neumann algebra,
which is the starting point of this paper.

It is possible for $(\Omega, \KB^\scs, \nu|_{\KB^\scs})$ to be non-localizable, even if  $(\Omega, \KB, \nu)$ is localizable.
Moreover, the existence of a metric preserving bijection from $L^p(\mu)_+^\s$ onto $L^p(\nu)_+^\s$ may not imply that $\KA$ and $\KB$ are regular set isomorphic.
The following gives such an example.

\medskip

\begin{eg}\label{eg:non-local}
Let $\Gamma$ be an uncountable set, $\KB$ be the collection of all subsets of $\Gamma$ and $\nu$ be the
counting measure on $\KB$.
Consider $\KA= \KB^\scs$.
Then $\KA$ consists of all subsets that are either countable or co-countable.
It is clear  that $(\Gamma, \KB, \nu)$ is localizable, while $(\Gamma, \KA, \nu|_{\KA})$ is not.
Moreover, we have $L^p(\nu) = L^p(\nu|_{\KA})$ and their elements have the same norm (in fact, both coincide with $\ell^p(\Gamma)$).
However, there is no regular set isomorphism between $\KA$ and $\KB$, as $\od{\KA}$ and $\od{\KB}$ have different cardinality.
\end{eg}

In some special cases, we have a better presentation of $\Phi$ than the one in Theorem \ref{thm:Lp}(c) above.
For these, we need some more notations.

Let $\Omega$ be a Hausdorff topological space.
Denote by $\KB_{\mathrm{Bo}}$ and  $\KB_{\mathrm{Ba}}$ the collections of all Borel subsets and of all Baire subsets of $\Omega$, respectively.
Let $(\Omega,\KA,\mu)$ be a measure space such that $\KB_{\mathrm{Bo}}\subseteq \KA$ and $A\in \KA$.
We say that $\mu$ is \emph{inner regular} on $A$ if
\begin{equation}\label{eqt:def-inner-reg}
\mu(A)=\sup \{\mu(K): K\subseteq A; \text{ $K$ is compact in $\Omega$}\}.
\end{equation}

Recall that a Borel measure space $(\Omega, \KB_{\mathrm{Bo}}, \mu)$ is \emph{completion regular}
if the completion of $(\Omega, \KB_{\mathrm{Bo}}, \mu)$ coincides with the completion of the Baire measure space $(\Omega, \KB_{\mathrm{Ba}}, \mu|_{\KB_{\mathrm{Ba}}})$.
Examples of completion regular  Borel measure spaces include the cases when  $\Omega$ is metrizable or when $\Omega$ is a second countable
locally compact Hausdorff space (see e.g., \cite[p.89]{Lessard94}).
We also recall that a $\sigma$-finite Baire measure $(\KB_{\mathrm{Ba}}, \mu)$ on $\Omega$ is said to be \emph{tight} if it is equivalent to a finite Baire measure $(\KB_{\mathrm{Ba}}, \lambda)$ on $\Omega$ satisfying: for every $\epsilon > 0$,
there is a compact subset $K\subseteq \Omega$ such that $\lambda^*(K)\geq \lambda(\Omega) - \epsilon$; here $\lambda^*$ is the outer measure associated with $\lambda$.
Note that if $\Omega$ is $\sigma$-compact, then all $\sigma$-finite Baire measures are tight (see e.g., \cite[p.83]{Lessard94}).

\begin{cor}\label{cor:LpP}
Let $(\Omega, \KA, \mu)$ and $(\Gamma, \KB, \nu)$ be $\sigma$-finite measure spaces such that $\Omega$ is a Hausdorff topological space and  $(\KA, \mu)$ is either
a tight Baire measure, or
a completion regular Borel measure with all Borel subsets with finite measures being inner regular.
Suppose that $\Phi: L^p(\mu)_+^\s \to L^p(\nu)_+^\s$ is a metric preserving  bijection.
Then there exist a  non-negative $\KB$-measurable function $h$ on $\Gamma$
 and a $\KB_0$-$\KA$ measurable bijection $\psi : \Gamma \to \Omega$, such that
$$\Phi(f) = h (f\circ\psi)   \quad (f\in L^p(\mu))
\quad \text{and} \quad \mu(A) = \int_{\psi^{-1}(A)} h^p d\nu  \quad (A\in \KA);$$
here, $\KB_0$ is the $\sigma$-algebra of the completion of $(\Gamma, \KB, \nu)$.
\end{cor}
\begin{proof}
By Theorem \ref{thm:Lp}, we know that $\Phi$ extends to an isometric order isomorphism $\bar{\Phi}$ from $L^p(\mu)$ onto $L^p(\nu)$.
The assertions are then consequences of the fact that  $\bar{\Phi}$ is a Lamperti operator, i.e., it sends disjoint functions to disjoint functions, together with \cite[Theorem 5.4]{Lessard94} and the proof of \cite[Corollary 5.6]{Lessard94}.
\end{proof}

\section{The case of $L^\infty(\mu)$ and the case of $C(X)$}\label{s:4}

Let $X$ be a compact Hausdorff space.
For $f\in C(X)_+^\s$, we define
$$
\rp(f) := \{x\in X: f(x) = 1\}\quad\text{and}\quad \rz(f):=\{x\in X: f(x)=0\},
$$
i.e., $\rp(f)$ is the set of peak points and $\rz(f)$ is the set of zeros, respectively, of $f$.
Moreover, we define
$$\CP:= \{f\in C(X)_+^\s: \rz(f) = \emptyset \}.$$
For every subset $\CS\subseteq C(X)_+^\s$, let
\begin{gather*}
\CS^\#:=\{g\in C(X)_+^\s: \|f-g\| <1, \text{ for all }g\in \CS \}.
\end{gather*}
For $E,F\subseteq X$, denote by
$$\CF_{F}^{E}:= \{g\in C(X)_+^\s: \rp(g)\subseteq E; \rz(g)\subseteq F\}$$
the collection of functions in $C(X)_+^\s$ with all peak points in $E$ and all zeros in $F$.

\begin{lem}\label{lem:S''}
(a) For a subset $\CS \subseteq C(X)_+^\s$, if we put $E_\CS:= {\bigcap}_{g\in \CS} X\setminus \rz(g)$ and $F_\CS:= {\bigcap}_{g\in \CS} X\setminus \rp(g)$, then
\begin{gather*}
\CS^\#= \CF_{F_\CS}^{E_\CS}.
\end{gather*}

\noindent
(b) Let $f\in C(X)_+^\s$.  We have $f\in \CP$
if and only if
 for every $\CS\subseteq C(X)_+^\s$ with $\CS^\#\neq \emptyset$, there is $g_0\in \CS^\#$ satisfying $\|f-g_0\| <1$.

\noindent
(c) If $g\in \CP$, then $\{g\}^{\#\#} = \big(\CF_{X\setminus \rp(g)}^X\big)^\# = \CF^{\rp(g)}_\emptyset$.

\noindent
(d) If $f,g\in \CP$ satisfying $\CF^X_{X\setminus \rp(f)}\subseteq \CF^X_{X\setminus \rp(g)}$, then $\rp(g)\subseteq \rp(f)$.
\end{lem}
\begin{proof}
(a) Suppose that $f\in \CS^\#$.
If there exists $x\in \rp(f)\setminus E_\CS$, then there is $g\in \CS$ with $x\in \rz(g)$.
We then have the contradiction that $\|f - g\| = 1$.
A similar contradiction will occur if $\rz(f)\setminus F_\CS\neq \emptyset$.
Conversely, let $f\in \CF_{F_\CS}^{E_\CS}$.
If there is $g\in \CS$ satisfying $\|f-g\| = 1$, then either $\rp(g)\cap \rz(f)\neq \emptyset$ or $\rz(g)\cap \rp(f)\neq \emptyset$.
However, any one of them   derives a contradiction.

\noindent
(b) Suppose that $\rz(f) =\emptyset$ and $\CS^\# \neq \emptyset$.
By part (a), we know that $\CS^\# = \CF_{F_\CS}^{E_\CS}$.
If $h\in \CS^\#$, then
$$g_0:=(1+h)/2\in \CF_{\emptyset}^{E_\CS} \subseteq \CS^\#,$$
and $\|f - g_0\| < 1$.
Conversely, suppose that $\rz(f)\neq \emptyset$.
For every $x\in X\setminus \rz(f)$, one can find $g_x\in C(X)_+^\s$ with $g_x(x)=0$ and $\rz(f)\subseteq \rp(g_x)$.
Hence, $\rz(f)\subseteq X\setminus \rz(g_x)$.
We set $\CS:= \{g_x: x\in X\setminus \rz(f)\}$.
Then $E_\CS = \rz(f)\neq \emptyset$ and $(2-f)/2\in \CF^{E_\CS}_\emptyset \subseteq \CF^{E_\CS}_{F_\CS} =\CS^\#$ (by part (a)).
However, for every $g\in \CS^\#\subseteq \CF^{E_\CS}_X$, one has $\|f-g\|=1$ (note that $\rp(g)\neq \emptyset$).

\noindent
(c) By part (a), one has $\{g\}^\# = \CF_{X\setminus \rp(g)}^{X}$.
Moreover, it is clear that $\CF_\emptyset^{\rp(g)}\subseteq \big(\CF_{X\setminus \rp(g)}^{X}\big)^{\#}$.
Conversely, suppose on the contrary that there exists $f\in \big(\CF_{X\setminus \rp(g)}^{X}\big)^{\#}$
such that $\rp(f)\not\subseteq \rp(g)$ or $\rz(f)\neq \emptyset$.
In the first case, we pick $x\in \rp(f)\setminus \rp(g)$ and $h_1\in C(X)_+^\s$ with $h_1(x)=0$ and $\rp(g)\subseteq \rp(h_1)$.
We then have the contradiction that $h_1\in \CF_{X\setminus \rp(g)}^{X}$ but $\|f - h_1\| = 1$.
In the second case, we pick $y\in \rz(f)$ and $h_2\in C(X)_+^\s$ with $h_2(y) = 1$.
Then $(1+h_2)/2 \in \CF_{\emptyset}^{X} \subseteq \CF_{X\setminus \rp(g)}^{X}$ but $\|f - (1+h_2)/2\| = 1$, which is a contradiction.

\noindent
(d) It follows from  part (c) that $\CF^{\rp(g)}_\emptyset\subseteq \CF^{\rp(f)}_\emptyset$.
Suppose on the contrary that there is $x\in \rp(g)\setminus \rp(f)$.
Consider $h\in C(X)_+^\s$ with $h(x) = 1$ and $\rp(f)\subseteq \rz(h)$.
Set $h_0:= (g+h)/2$.
Then we arrive at a contradiction that $h_0\in \CF^{\rp(g)}_\emptyset\setminus \CF^{\rp(f)}_\emptyset$.
\end{proof}

Recall that $Z\subseteq X$ is a \emph{zero set} of $X$ if $Z = \rz(f)$ for some $f\in \CPO$; equivalently, $Z = \rp(g)$ for some $g\in \CP$ (e.g., by setting $g:= (2-f)/2$).

\begin{thm}\label{thm:C(X)}
Let $X$ and $Y$ be compact Hausdorff spaces.
Suppose that $\Phi:   C(X)^\s_+ \to  C(Y)^\s_+$ is a metric preserving bijection.
Then $\Phi$ can be extended uniquely to a $^*$-isomorphism from $C(X)$ onto $C(Y)$.
More precisely, there is a homeomorphism $\sigma: Y\to X$ such that $\Phi(f)=f\circ\sigma$\ \ ($f\in C(X)_+^\s$).
\end{thm}
\begin{proof}
Observe that $\Phi(\CS^\#) = \Phi(\CS)^\#$,  for any $\CS\subseteq \CPO$.
Thus, by parts (a) and (b) of Lemma \ref{lem:S''},
$$\rz(\Phi(f)) = \emptyset, \quad \{f\}^\# = \CF^X_{X\setminus \rp(f)} \quad \text{and} \quad \Phi(\{f\}^\#) = \{\Phi(f)\}^\# = \CF^X_{X\setminus \rp(\Phi(f))} \qquad (f\in \CP).$$
If $f,g\in \CP$ satisfying $\rp(g) \subseteq \rp(f)$, then $\{f\}^\# \subseteq \{g\}^\#$
and we have $\CF^X_{X\setminus \rp(\Phi(f))} \subseteq \CF^X_{X\setminus \rp(\Phi(g))}$, which gives
$\rp(\Phi(g)) \subseteq \rp(\Phi(f))$ (because of Lemma \ref{lem:S''}(d)).
Consequently, $\Phi$ induces a map $\Psi: Z(X)\to Z(Y)$ between the collections $Z(X)$ and $Z(Y)$ of zero sets of $X$ and $Y$,
respectively, satisfying that
\begin{equation}\label{eqt:Psi-Phi}
\Psi(\rp(f)) = \rp(\Phi(f)) \qquad (f\in \CP),
\end{equation}
and that
$$
\Psi(A)\subseteq \Psi(B) \quad \text{whenever }A,B\in Z(X) \text{ with }A\subseteq B.
$$
By considering $\Phi^{-1}$, we know that $\Psi$ is an order isomorphism from $Z(X)$ onto $Z(Y)$.

Suppose that $f_1,...,f_n\in \CP$ with $\bigcap_{i=1}^n \rp(f_i)\neq \emptyset$.
If we define $f:= \sum_{i=1}^n f_i/n$, then $f\in \CP$ and $\rp(f) = \bigcap_{i=1}^n \rp(f_i)$.
For $i= 1,...,n$, it follows from $\{f_i\}^\# \subseteq \{f\}^\#$ that $\{\Phi(f_i)\}^\# \subseteq \{\Phi(f)\}^\#$, and hence $\Psi(\rp(f))\subseteq \Psi(\rp(f_i))$ (again because of Lemma \ref{lem:S''}(d)).
This shows that $\bigcap_{i=1}^n\Psi(\rp(f_i))\supseteq \Psi(\rp(f)) \neq \emptyset$.

Fix $x\in X$ and consider
$$\KF_x:= \{\rp(f): f\in \CP; x\in \rp(f)\}.$$
The above and the finite intersection property tells us that there exists $y\in \bigcap \Psi(\KF_x)$.
It follows from Relation \eqref{eqt:Psi-Phi} that  $\Psi(\KF_x)\subseteq \KF_y$.
Applying the same arguments to $\Phi^{-1}$ and $\KF_y$, we obtain $z\in X$ such that $\Psi^{-1}(\KF_y) \subseteq \KF_z$.
Therefore, $\KF_x\subseteq \KF_z$, which implies $x = z$ (otherwise, there exists $g\in \CPO$ with $g(x) = 1$ and $g(z) = 0$, which gives the contradiction that $\rp((1+g)/2)\in \KF_x\setminus \KF_z$).
We thus conclude that $\Psi(\KF_x) = \KF_y$.
This produces a bijection $\tau: X\to Y$ satisfying
\begin{equation}\label{eqt:tau}
\Psi(\KF_x) = \KF_{\tau(x)} \qquad (x\in X).
\end{equation}

Consider $g\in \CP$.
It follows from Relation \eqref{eqt:tau} that $\{\tau(x): x\in \rp(g) \}\subseteq \Psi(\rp(g))$.
Conversely, for any $y\in \Psi(\rp(g))$, if we set $x:=\tau^{-1}(y)$, then $\Psi(\KF_x) = \KF_y$ and
one knows from $\Psi(\rp(g)) \in \KF_y$ that $\rp(g)\in \KF_x$; i.e, $x\in \rp(g)$.
These show that
$$\Psi(\rp(g)) = \{\tau(x): x\in \rp(g) \} \qquad (g\in \CP).$$
In other words, the bijection $\tau$ preserves zero sets in both directions.
Since in a completely regular space, zero sets form a basis for closed subsets (\cite[Theorem 3.2]{GJ76}),
$\tau$ is a homeomorphism.

Define a map $\Lambda: \CPO \to \CPO$ by $\Lambda(f) := \Phi(f)\circ\tau$\ \ ($f\in \CPO$).
For $x\in X$, we put
$$\CP^x := \{f\in \CP: \rp(f) \in \KF_x\} = \{f\in \CP: f(x) =1\}.$$
We see from Relations \eqref{eqt:Psi-Phi} and \eqref{eqt:tau} that $\Phi(\CP^x) = \CP^{\tau(x)}$ for each $x\in X$.
Therefore, $\Lambda$ is a metric preserving bijection satisfying $\Lambda(\CP^x) = \CP^x$ ($x\in X$).

	Fix  $f\in C(X)^\s_+$ and set $\dist(f, \CP^x):= \inf\{\|f-g\|: g\in \CP^x\}$.
	For every $t\in (0,1]$ and $x\in X$, we have
	$$
	f(x)\leq 1- t \quad\text{if and only if}\quad \dist(f, \CP^x)\geq t.
	$$
	In fact, it is clear that $f(x)\leq 1-t$ implies $\dist(f, \CP^x)\geq t$.
	Conversely, suppose on the contrary that $\dist(f, \CP^x)\geq t$ but $f(x) > 1- t + \epsilon$ for some $\epsilon \in (0,t)$.
	Set $g_0(y) := \min \{1, f(y)+t-\epsilon\}$ ($y\in X$).
	Then $g_0\in \CP^x$ but $\|f-g_0\| \leq t-\epsilon$, which contradicts to the fact that $\dist(f, \CP^x)\geq t$.

	Since $\dist(\Lambda(f), \CP^x)=\dist(f, \CP^x)$, one obtains
	$$
	\{x\in X: f(x)\leq s \} = \{x\in X: \Lambda(f)(x)\leq s\} \qquad (s\in [0,1)).
	$$
	Consequently, $\Lambda(f)=f$.
	Finally, the assertion in the statement follows if we set $\sigma=\tau^{-1}$.
\end{proof}

For a measure space $(\Omega, \KA, \mu)$, denote by $M^\infty(\mu)$ the algebra of all  bounded  $\KA$-measurable functions.
If $N^\infty(\mu):= \big\{f\in M^\infty(\mu): \mu(\{x\in \Omega: f(x)\neq 0 \}) = 0 \big\}$, then
$$L^\infty(\mu):= M^\infty(\mu)/N^\infty(\mu)$$
is a Banach space under the norm induced by the essential suprema.

\begin{thm}\label{thm:L-infty}
	Let $(\Omega, \KA, \mu)$ and $(\Gamma, \KB, \nu)$ be measure spaces.
	Suppose that $\Phi: L^\infty(\mu)_+^\s \to L^\infty(\nu)_+^\s$ is a metric preserving bijection.
	Then $\Phi$ extends to a $^*$-isomorphism from $L^\infty(\mu)$ onto $L^\infty(\nu)$.
More precisely, there is a regular set isomorphism $\Lambda:\KA \to \KB$ such that $\Phi(g) = \widehat{\Lambda}(g)$ ($g\in L^\infty(\mu)_+^\s$).
\end{thm}
\begin{proof}
Since $L^\infty(\mu)$ and $L^\infty(\nu)$ are unital commutative $C^*$-algebras, there exists compact Hausdorff spaces $X$ and $Y$ such that $L^\infty(\mu)\cong C(X)$ and $L^\infty(\nu)\cong C(Y)$.
By Theorem \ref{thm:C(X)}, $\Phi$ can be extended to a $^*$-isomorphism from $L^\infty(\mu)$ onto $L^\infty(\nu)$.
The second conclusion follows from a well-known fact.
\end{proof}

In the case when the measure spaces are  Radon measure spaces, we have a better presentation of $\Phi$ than the one in Theorem \ref{thm:L-infty}.
In the following, $\Omega$ is a locally compact Hausdorff space.

\medskip

Let $I$ be a positive linear functional on the ordered vector	space,
$C_c(\Omega)$, of all continuous functions on $\Omega$ with compact supports.
There is a standard procedure to construct  from $I$, a complete measure $\mu$ on $\Omega$,  whose defining $\sigma$-algebra $\KA$ contains $\KB_{\mathrm{Bo}}$, satisfying certain nice properties  (see e.g. \cite[Section 7.5]{Cohn} or \cite[Theorem 2.14]{Rudin}).
In particular, $\mu$ is finite on all compact subsets, inner regular on all open subsets (see \eqref{eqt:def-inner-reg}), and satisfies $\mu(A)=\inf\{\mu(U): A\subseteq U; \text{ $U$ is open in $\Omega$}\}$ for every $A\in \KA$.
The measure $(\KA, \mu)$ constructed from $I$ in this way is called a  \emph{Radon measure}.
Note that there are several different definitions for Radon measures in the literatures, which are closely related.
For our definition in the above, a Radon measure $(\KA,\mu)$ will be the Caratheodory extension of $(\KB_{\mathrm{Bo}}, \mu|_{\KB_{\mathrm{Bo}}})$ (see e.g., \cite{LNW-study-note} for a quick survey).

Let $(\KA, \mu)$ be a Radon measure on $\Omega$.
We set
$$\mu_\smf(E):=\sup \{\mu(A): A\in \KA^\f; A\subseteq E \} \qquad (E\in \KA).$$
Then $(\KA, \mu_\smf)$ is a semi-finite measure on $\Omega$.
Conversely, when $(\KA,\mu)$ is semi-finite,
one has $\mu_\smf = \mu$.

Let $N^\infty_\loc(\mu):= \big\{f\in M^\infty(\mu): \mu_\smf(\{x\in \Omega: f(x)\neq 0 \})=0 \big\}$.
By \cite[Theorem 3.4.1]{Cohn}, the quotient space
$L^\infty_\loc(\mu):= M^\infty(\mu)/N^\infty_\loc(\mu)$
is a Banach space (note that our $L^\infty_\loc(\mu)$ coincides with $L^\infty(\Omega, \KA, \mu)$ in \cite{Cohn}).
In fact, we have
\begin{equation}\label{eqt:L-infty-loc}
L^\infty_\loc(\mu) = L^\infty(\mu_\smf).
\end{equation}
It was shown in \cite[Theorem 7.5.4]{Cohn}
that if $(\KA, \mu)$ is a Radon measure on $\Omega$, then $L^1(\mu)^* = L^\infty_\loc(\mu)$.
As a result, $L^\infty_\loc(\mu)$ is a commutative von Neumann algebra.
In the following, we consider $Q:M^\infty(\mu)\to L^\infty_\loc(\mu)$ to be the quotient map.

\begin{cor}\label{cor:Radon}
	Let $(\KA, \mu)$ and $(\KB, \nu)$ be Radon
 measures over locally compact spaces $\Omega$ and $\Gamma$, respectively.
	Suppose that $\Phi: L^\infty_\loc(\mu)_+^\s \to L^\infty_\loc(\nu)_+^\s$ is a metric preserving bijection.

\smnoind
(a)	There is a
map $\psi: \Gamma\to \Omega$ implementing $\Phi$, i.e. the following conditions are satisfied:
	\begin{enumerate}[\ \ 1.]
		\item if $f\in M^\infty(\mu)$ then $f\circ \psi$ belongs to $M^\infty(\nu)$;
		\item $\Phi(Q(f)) = Q(f\circ\psi)$, for each $f\in M^\infty(\mu)_+$ with $\|Q(f)\| = 1$.
	\end{enumerate}

\smnoind
(b) Suppose that $\Omega$ and $\Gamma$ are metrizable.
Let $\tau:\Omega\to \Gamma$ be a map implementing $\Phi^{-1}$ as in part (a).
Then there are $\Omega'\in \KA^0$ and $\Gamma'\in \KB^0$ such that $\psi$ is a bijection from $\Gamma \setminus \Gamma'$ onto $\Omega\setminus \Omega'$ with $\psi|_{\Gamma \setminus \Gamma'}^{-1} = \tau|_{\Omega\setminus \Omega'}$.
\end{cor}
\begin{proof}
(a) By Theorem \ref{thm:L-infty} and Relation \eqref{eqt:L-infty-loc}, $\Phi$ can be extended to a $^*$-isomorphism from
the commutative von Neumann algebra $L^\infty_\loc(\mu)$ onto the commutative von Neumann algebra $L^\infty_\loc(\nu)$.
The conclusion now follows from a standard lifting result (see e.g., Theorem 1 in Section X.3 of \cite{TulceaBook69}, or the main result of \cite{VW69} for the precise form that we needed).

\smnoind
(b) The assertion follows  from Proposition 5 in Section X.4 of \cite{TulceaBook69}.
\end{proof}

\section*{Acknowledgement}

The first author is supported by Hong Kong RGC Research Grant (2130501) and the second author is supported by National Natural Science Foundation of China grant (11871285).
The third author would like to thank the colleagues in
Tiangong University, Nankai University and the Chinese University of Hong Kong
for their warm hospitality  during his visits there when this work started.

\bibliographystyle{plain}

\end{document}